\documentclass[12pt]{amsart}
\usepackage{url, 
	amssymb,setspace, mathrsfs,fontenc, comment}
\usepackage[alphabetic]{amsrefs}
\usepackage{fullpage} 
\usepackage{color}
\usepackage{tikz-cd}
\usepackage[all]{xy}
\usepackage{amsmath}

\usepackage{url, 
	amssymb,setspace, mathrsfs,fontenc}
\usepackage{amsrefs}
\usepackage{graphicx}
\usepackage[mathcal]{euscript}
\usepackage{verbatim}
\usepackage{hyperref}
\hypersetup{backref=true}
\usepackage{mathtools}
\usepackage{tikz}
\usetikzlibrary{chains}

\tikzset{node distance=2em, ch/.style={circle,draw,on chain,inner sep=2pt},chj/.style={ch,join},every path/.style={shorten >=4pt,shorten <=4pt},line width=1pt,baseline=-1ex}

\newtheorem{thm}{Theorem}
\newtheorem{lem}[thm]{Lemma}

\newtheorem{conj}[thm]{Conjecture}

\newtheorem{defe}[thm]{Definition}

\theoremstyle{remark}
\newtheorem{rem}[thm]{Remark}

\DefineSimpleKey{bib}{myurl}
\newcommand\myurl[1]{\url{#1}}
\BibSpec{webpage}{
	+{}{\PrintAuthors} {author}
	+{,}{ \textit} {title}
	+{}{ \parenthesize} {date}
	+{,}{ \myurl} {myurl}
}
\usepackage{arydshln}

\newcommand{\nc}{\newcommand}

\nc{\ssec}{\subsection}

\nc{\on}{\operatorname}

\nc{\sE}{\mathscr{E}}
\nc{\sF}{\mathscr{F}}
\nc{\sL}{\mathscr{L}}
\nc{\sD}{\mathscr{D}}
\nc{\sA}{\mathscr{A}}

\nc{\cC}{\mathcal{C}}
\nc{\cG}{\mathcal{G}}
\nc{\cV}{\mathcal{V}}
\nc{\CB}{\mathcal{B}}
\nc {\K}{\mathcal{K}}

\nc{\cE} {\mathcal{E}}
\nc{\Kl}{\mathrm{Kl}}
\nc{\cO}{\mathcal{O}}
\nc{\cF}{\mathcal{F}}
\nc{\cZ}{\mathcal{Z}}
\nc{\bcZ}{\overline{\mathcal{Z}}}
\nc{\bcB}{\overline{\mathcal{B}}}
\nc{\cD}{\mathcal{D}}
\nc{\cDt}{\mathcal{D}^\times}
\nc{\cH}{\mathcal{H}}
\nc{\bZ}{\mathbb{Z}}
\nc{\bH}{\mathbb{H}}
\nc{\bQ}{\mathbb{Q}}
\nc{\bR}{\mathbb{R}}
\nc{\bC}{\mathbb{C}}
\nc{\bQl}{\overline{\mathbb{Q}}_\ell}
\nc{\bQlt}{\bQl^\times} 
\nc{\FG}{\mathrm{FG}}
\nc{\dR}{\mathrm{dR}}
\nc{\dv}{\dot{v}}
\nc{\du}{\dot{u}}
\nc{\bbR}{\mathbb{R}}

\nc{\uG}{\underline{G}}
\nc{\uc}{\underline{c}}
\nc{\uu}{\underline{u}}
\nc{\cU}{\mathcal{U}}
\nc{\rat}{\mathrm{rat}}
\nc{\Hyp}{\mathrm{Hyp}}
\nc{\Lie}{\mathrm{Lie}}
\nc{\ctheta}{\check{\theta}}
\nc{\nil}{\mathrm{nil}}
\nc{\bLX}{\overline{LX}}
\nc{\bOmega}{\overline{\Omega}}
\nc{\tOmega}{\widetilde{\Omega}}

\nc{\fF}{\mathfrak{F}}
\nc{\fB}{\mathfrak{B}}
\nc{\fZ}{\mathfrak{Z}}
\nc{\fx}{\mathfrak{x}}
\nc{\fy}{\mathfrak{y}}
\nc{\fb}{\mathfrak{b}}
\nc{\fk}{\mathfrak{k}}
\nc{\fI}{\mathfrak{i}}
\nc{\fj}{\mathfrak{j}}
\nc{\fg}{\mathfrak{g}}
\nc{\fu}{\mathfrak{u}}
\nc{\fl}{\mathfrak{l}}
\nc{\fn}{\mathfrak{n}}
\nc{\cP}{\mathcal{P}}
\nc{\cQ}{\mathcal{Q}}
\nc{\ft}{\mathfrak{t}}
\nc{\fz}{\mathfrak{z}}
\nc{\fc}{\mathfrak{c}}
\nc{\cfc}{\check{\mathfrak{c}}}
\nc{\fh}{\mathfrak{h}}
\nc{\fp}{\mathfrak{p}}
\nc{\cfp}{\mathring{\mathfrak{p}}}
\nc{\bone}{\mathbf{1}}
\nc{\tg}{\mathtt{g}}
\nc{\hfg}{\widehat{\fg}}
\nc{\ch}{\check{\fh}}
\nc{\hP}{\hat{P}}
\nc{\hg}{\widehat{\mathfrak{g}}}
\nc{\gO}{\mathfrak{g}[\![t]\!]}
\nc{\Ug}{\widehat{U}(\mathfrak{g})}
\nc{\dl}{/\!\!/}

\nc{\bGm}{\mathbb{G}_m}
\nc{\bGa}{\mathbb{G}_a}
\nc{\bL}{\mathbf{L}}
\nc{\bK}{\mathbf{K}}
\nc{\bJ}{\mathbf{J}}
\nc{\bI}{\mathbf{I}}
\nc{\bV}{\mathbb{V}}
\nc{\bM}{\mathbb{M}}
\nc{\bP}{\mathbb{P}}
\nc{\bA}{\mathbb{A}}
\nc{\bN}{\mathbb{N}}

\nc {\Q}{\mathrm{Q}}
\nc{\diag}{\mathrm{diag}}
\nc{\diff}{\mathrm{diff}}
\nc{\ev}{\mathrm{ev}}
\nc{\Res}{\mathrm{Res}}
\nc{\Fl}{\mathcal{F}\ell}
\nc{\Ad}{\mathrm{Ad}}
\nc{\ad}{\mathrm{ad}}
\nc{\pr}{\mathrm{pr}}
\nc{\Sl}{\mathfrak{sl}}
\nc{\gl}{\mathfrak{gl}}
\nc{\ra}{\rightarrow}
\nc{\tra}{\twoheadrightarrow}
\nc{\hra}{\hookrightarrow}
\nc{\quo}{\mathopen{ /\!/}}
\nc{\GL}{\mathrm{GL}}
\nc{\SL}{\mathrm{SL}}
\nc{\Sp}{\mathrm{Sp}}
\nc{\SO}{\mathrm{SO}}
\nc{\so}{\mathfrak{so}}
\nc{\PGL}{\mathrm{PGL}}
\nc{\Bun}{\mathrm{Bun}}
\nc{\supp}{\mathrm{supp}}
\nc{\bgamma}{\bar{\gamma}}
\nc{\ab}{\mathrm{ab}}
\nc{\td}{\mathrm{d}}
\nc{\Ht}{\mathrm{ht}}
\nc{\tX}{\tilde{X}}
\nc{\rW}{\mathrm{W}}

\nc         {\rar}[1]       {\stackrel{#1}{\longrightarrow}}

\nc{\fa}{\mathfrak{a}}
\nc{\Hit}{\mathrm{Hit}}

\nc{\RS}{\mathrm{RS}}
\nc{\Loc}{\mathrm{Loc}}
\nc{\tLoc}{\widetilde{\mathrm{Loc}}}
\nc{\reg}{\mathrm{reg}}
\nc{\im}{\mathrm{Im}}

\nc{\tp}{\mathfrak{p}}
\nc{\cA}{\mathcal{A}}
\nc{\cY}{\mathcal{Y}}

\nc{\opp}{\mathrm{opp}}
\nc{\Ind}{\mathrm{Ind}}
\nc{\sAn}{\mathrm{can}}
\nc{\Lg}{\check{\fg}}
\nc{\cDelta}{\check{\Delta}}
\nc{\cPhi}{\check{\Phi}}
\nc{\LV}{\check{V}}
\nc{\Lh}{\check{h}}
\nc{\LG}{\check{G}}
\nc{\cT}{\check{T}}
\nc{\ct}{\check{\ft}}
\nc{\cB}{\check{B}}
\nc{\cb}{\check{\fb}}
\nc{\cN}{\check{N}}
\nc{\sN}{\mathcal{N}}
\nc{\cn}{\check{\fn}}
\nc{\Spec}{\mathrm{Spec}}
\nc{\End}{\mathrm{End}}
\nc{\crho}{\check{\rho}}
\nc{\clambda}{\check{\lambda}}
\nc{\rX}{\mathring{X}}
\nc{\ru}{\mathring{u}}

\nc{\sW}{\mathscr{W}}
\nc{\sH}{\mathscr{H}}
\nc{\sV}{\mathscr{V}}
\nc{\geom}{\mathrm{geom}}
\nc{\Irr}{\mathrm{Irr}}
\nc{\fm}{\mathfrak{m}}
\nc{\aff}{\mathrm{aff}}
\nc{\Aut}{\mathrm{Aut}}
\nc{\cJ}{\mathcal{J}}
\nc{\fs}{\mathfrak{s}}
\nc{\Stab}{\mathrm{Stab}}
\nc{\st}{\mathrm{st}}
\nc{\tw}{{\widetilde{w}}}
\nc{\gen}{\mathrm{gen}}
\nc{\genn}{\mathrm{genn}}
\nc{\sss}{\mathrm{ss}}
\nc{\fsp}{\mathfrak{sp}}
\nc{\Hom}{\mathrm{Hom}}
\nc{\bm}{\mathbf{m}}
\nc{\HG}{\mathcal{HG}}
\nc{\Gal}{\mathrm{Gal}}
\nc{\Sym}{\mathrm{Sym}}
\nc{\rank}{\mathrm{rank}}

\nc{\calX}{\mathcal{X}}
\nc{\tP}{\mathtt{P}}
\nc{\tL}{\mathtt{L}}
\nc{\tU}{\mathtt{U}}

\nc{\tW}{\widetilde{W}}
\nc{\tdb}{\tilde{b}}
\nc{\tdd}{\tilde{d}}
\nc{\tv}{\tilde{v}}
\nc{\Hk}{\on{Hk}}
\nc{\cL}{\mathcal{L}}
\nc{\talpha}{\widetilde{\alpha}}
\nc{\tQ}{{\widetilde{Q}}}
\nc{\ochi}{\overline{\chi}}
\nc{\tdelta}{\widetilde{\Delta}}
\nc{\wt}{\mathrm{wt}}
\nc{\fQ}{\mathfrak{Q}}
\nc{\bbP}{\mathbb{P}}

\nc{\inv}{\mathrm{inv}}
\nc{\Rep}{\mathrm{Rep}}
\nc{\Conn}{\mathrm{Conn}}
\nc{\Hecke}{\mathrm{Hecke}}
\nc{\Gr}{\mathrm{Gr}}
\nc{\GR}{\mathrm{GR}}
\nc{\IC}{\mathrm{IC}}
\nc{\Std}{\mathrm{Std}} 
\nc{\Db}{\mathrm{D}^{\mathrm{b}}}
\nc{\tr}{\mathrm{tr}}
\nc{\gr}{\mathrm{gr}}
\nc{\tmin}{\mathrm{min}}
\nc{\Fun}{\mathrm{Fun}~}
\nc{\Temp}{\mathrm{Temp}}

\nc{\bbA}{\mathbb{A}}
\nc{\mO}{\mathrm{O}}

\newcommand{\quash}[1]{}

\AtEndDocument{\bigskip{\footnotesize

\textsc{Tsao-Hsien Chen, School of Mathematics, University of Minnesota, Twin cities, Minneapolis, MN 55455 } \par
\textit{E-mail address}: \texttt{chenth@umn.edu} \par
		
\textsc{Lingfei Yi, Shanghai Center for Mathematical Sciences, Fudan University, Shanghai 200438, China} \par
\textit{E-mail address}: \texttt{yilingfei@fudan.edu.cn} \par
}}

\begin{document} 
\renewcommand{\thepart}{\Roman{part}}

\renewcommand{\partname}{\hspace*{20mm} Part}

\title{Loop spaces, Twistor $\mathbb P^1$ and tempiric parameters} 
\date{\today}

 \author{Tsao-Hsien Chen and Lingfei Yi}

\dedicatory{}

\maketitle
\begin{abstract}

We establish bijections among spherical orbits on loop symmetric spaces, principal bundles on the twistor $\mathbb P^1$, and tempiric Langlands parameters. Using Vogan's theory of minimal $K$-types, we further establish bijections among irreducible equivariant local systems on loop symmetric spaces, irreducible local systems on principal bundles on the twistor $\mathbb P^1$, and irreducible representations of all strong inner forms of the dual symmetric subgroup. The latter bijections may be viewed as a combinatorial shadow of geometric Langlands duality for real reductive groups.
    
\end{abstract}
\setcounter{tocdepth}{1}
\tableofcontents

\section{Introduction}
Let $G$ be a complex connected reductive group, and let $\check G$ denote its complex Langlands dual group. 
Set $F=\bC((t))$ and $\cO=\bC[[t]]$ the formal Laurent and Taylor series rings.
A starting point for geometric Langlands duality for complex reductive groups is the collection of bijections
\begin{equation}\label{complex groups}
    \xymatrix{G(\cO)\backslash G(F)/G(\cO)\ar[rr]\ar[dr]&&\on{Irr}(\check G)\ar[dl]\ar[ll]\\
&\Bun_G(\bP^1)(\bC)\ar[ul]\ar[ur]&}
\end{equation}
among the set $G(\cO)\backslash G(F)/G(\cO)$ of $G(\cO)\times G(\cO)$-spherical orbits on $G(F)$, the set $\Bun_G(\bP^1)(\bC)$  of isomorphism classes of $G$-bundles on the complex projective line $\bP^1$, and the set $\on{Irr}(\check G)$ of irreducible finite dimensional complex representations of $\check G$. The horizontal bijection follows from the Cartan decomposition for loop groups together with the highest-weight classification of irreducible representations, while the bijection between the left and bottom terms follows from the Birkhoff decomposition for loop groups and Grothendieck's classification of principal bundles on $\bP^1$.

In this note, we establish a real-group analogue of these bijections. To state our results, we first recall some notation and constructions from Lie theory. For simplicity, we assume throughout 
the introduction
that $G$ is simply connected.
Let $\theta_0$ be a pinned involution 
of $G$ and 
let 
$\eta_0$ be the real form of $G$
such that $\theta_0\circ\eta_0=\eta_0\circ\theta_0$ is a compact real form. 
Consider the loop symmetric space
\[X^\theta(F)=\{\gamma\in G(F)|\gamma\theta(\gamma)=e\}\]
where the involution $\theta:G(F)\to G(F)$ is defined by $\theta(\gamma(t))=\theta_0(\gamma(-t))$.
The arc group $G(\cO)$ acts on $X^\theta(F)$ by the $\theta$-conjugation $g\cdot\gamma=g\gamma\theta(g)^{-1}$.
On the other hand,  the 
antipodal 
conjugation
$c:\bP^1\to\bP^1, c(t)=-\bar t^{-1}$ together with the real form 
$\eta_0$ of $G$, gives rise to a real form 
\[\Bun_G(\bP^1)_\eta\] of $\Bun_G(\bP^1)$, which we refer to as  the stack of $G$-bundles on the twistor $\bP^1$. Our key
observation is the 
following connection of the above two sets with the so called 
tempiric Langlands parameters introduced and studied in \cite{AA,ABV}.
Let $\Gamma=\Gal(\bC/\bR)$ be the Galois group of $\bR$.
Following the definition in \emph{loc. cit.}, a Langlands parameter 
$\phi:\rW_\bR\to G^\Gamma=G\rtimes_{\theta_0}\Gamma$   is called tempiric if its restriction 
to  the factor $\bR_{>0}\subset\rW_\bR=\bC^\times\sqcup\bC^\times j$
is trivial. Denote by 
\[\Temp(G,\theta_0)\]
the set of tempiric Langlands paramaters. The group
$G$ acts naturally on $\Temp(G,\theta_0)$ by conjugation, and
two tempiric parameters are said to be equivalent  if they belong to the same $G$-orbit.

Here is our  main result:

\begin{thm}\label{main 1:intro}
There are natural bijections 
\[
    \xymatrix{G(\cO)\backslash X^\theta(F)\ar[rr]\ar[dr]&&G\backslash\Temp(G,\theta_0)\ar[dl]\ar[ll]\\
&\Bun_G(\bP^1)_{\eta}(\bbR)\ar[ul]\ar[ur]&}
\]
among the following sets:
\begin{enumerate}
    \item $G(\cO)\backslash X^\theta(F)$,   the set of
$G(\cO)$-orbits on $X^\theta(F)$,
 \item $\Bun_G(\bP^1)_{\eta}(\bbR)$,  the set
 of isomorphism classes of 
 $G$-bundles on the twistor $\bP^1$,
\item $G\backslash\Temp(G,\theta_0)$, the set of equivalence classes of tempiric Langlands parameters.
\end{enumerate}
Furthermore, suppose that
$\gamma\in X^\theta(F)$, $\phi\in\Temp(G,\theta_0)$, and $\cE\in\Bun_G(\bP^1)_{\eta}(\bbR)$
correspond to one another under these bijections. Then there are natural isomorphisms
\[\pi_0(Z_{G(\cO)}(\gamma))\cong\pi_0(Z_G(\phi))\cong\pi_0(\Aut(\cE))\]
between the component groups of the corresponding stabilizer and automorphism groups.
\end{thm}

The theorem follows by combining Theorem \ref{T=X} with Lemma \ref{review of Matsuki}. In fact, the bijection between the left and bottom terms, which is the content of Lemma \ref{review of Matsuki}, is precisely the Matsuki duality for loop groups established in \cite{CY}. Thus, the main new ingredient is the horizontal bijection, which is the content of Theorem \ref{K types}.

Vogan's theory of minimal $K$-types \cite{V} gives a bijection
\[\xymatrix{
\on{Loc}_G(\Temp(G,\theta_0))\ar[r]&\bigsqcup_{_{i\in I}}\on{Irr}(\check K_i)\ar[l]}\]
between the set
 $\on{Loc}_G(\Temp(G,\theta_0))$ of irreducible $G$-equivariant local systems on $G$-orbits in $\Temp(G,\theta_0)$
and the disjoint union of the sets
$\on{Irr}(\check K_i)$  of irreducible finite dimensional complex representations of the dual symmetric subgroups 
$\check K_i=\check G^{\check\theta_i}$, where $\check\theta_i,i\in I$ ranges over the  inner forms of the dual involution $\check\theta=-w_0\circ\check\theta_0$.
Combining with Theorem \ref{main 1:intro}, we obtain:

\begin{thm}\label{main 2:into}
There are bijections 
\begin{equation}\label{real groups}
    \xymatrix{\on{Loc}_{G(\cO)}(X^\theta(F))\ar[rr]\ar[dr]&&\bigsqcup_{_{i\in I}}\on{Irr}(\check K_i)\ar[dl]\ar[ll]\\
&\on{Loc}(\Bun_G(\bP^1)_{\eta}(\bbR))\ar[ul]\ar[ur]&}
\end{equation}
among the following sets:
\begin{enumerate}
    \item $\on{Loc}_{G(\cO)}(X^\theta(F))$,  the set of irreducible
   $G(\cO)$-equivariant local systems on 
 $G(\cO)$-orbits in $X^\theta(F)$,
 \item $\on{Loc}(\Bun_G(\bP^1)_{\eta}(\bbR))$,  the set
of  pairs $(\cE,\chi)$, where $\cE\in \Bun_G(\bP^1)_{\eta}(\bbR)$ 
and $\chi\in\on{Irr}(\pi_0(\Aut(\cE)))$ is an irreducible
complex representation of $\pi_0(\Aut(\cE))$,
\item $\bigsqcup_{_{i\in I}}\on{Irr}(\check K_i)$, 
the disjoint union of the sets
of irreducible finite dimensional complex representations 
of the  symmetric subgroups 
$\check K_i=\check G^{\check\theta_i}$
where \(\check\theta_i\) ranges over the  inner forms of the dual involution $\check\theta=-w_0\circ\check\theta_0$.
\end{enumerate}

\end{thm}

The theorem is restated as Theorem \ref{K types}. In view of the bijections~\eqref{complex groups}, we regard the bijections~\eqref{real groups} as a combinatorial shadow of geometric Langlands duality for real reductive groups.

Recall the celebrated geometric Satake equivalence 
\[\on{Perv}(G(\cO)\backslash G(F)/G(\cO))\cong\on{Rep}(\check G)\]
which upgrades the horizontal bijection in~\eqref{complex groups} to a tensor equivalence between the category of $G(\cO)\times G(\cO)$-equivariant perverse sheaves on $G(F)$ and the category $\on{Rep}(\check G)$ of finite-dimensional complex representations of the dual group $\check G$.
Consider 
 the abelian category
$\mathrm{Perv}(G^{}(\cO)\backslash X
^\theta(F))$
of $G(\cO)$-equivariant perverse sheaves on $X^\theta(F)$ and the direct sum category 
$\bigoplus_{i\in I}\on{Rep}(\check K_i)$ of finite dimensional complex representations of the dual symmetric subgroups $\check K_i$.
Motivated by the results of this paper and the recent work \cite{C} on Relative Langlands duality for loop symmetric spaces, we propose a real-group analogue of the geometric Satake equivalence, upgrading the horizontal bijection in ~\eqref{main 2:into} to an equivalence of module categories over $\on{Rep}(\check G)$:

\begin{conj}
There is an equivalence of abelian categories
\[\mathrm{Perv}(G^{}(\cO)\backslash X
^\theta(F))\cong\bigoplus_{i\in I}\on{Rep}(\check K_i)\]
such that the induced map between irreducible objects is given by the bijection~\eqref{real groups}.
Furthermore, the equivalence intertwines the Hecke action of 
$\on{Perv}(G(\cO)\backslash G(F)/G(\cO))$ and the restriction functor $\on{Rep}(\check G)\to\on{Rep}(\check K_i)$ under the geometric Satake equvalence.
\end{conj}

We refer to 
Section \ref{conj} to more 
detailed explanation of the conjecture, including generalizations 
to general reductive groups and strong involutions.

\subsection*{Acknowledgements} 
The research of
T.-H.~Chen is supported by NSF grant DMS-2143722 and Simons Fellowships.
Lingfei Yi is supported by the Grant No. JIH1414062Y of Fudan University.

\section{Local systems}
Let $G$ be a Lie group acting on a set $X$, and
 denote by $G\backslash X$ the set of $G$-orbits on $X$.

Assume that, for every $x\in X$, the
stabilier $Z_G(x)=\{g\in G|g\cdot x=x\}$ is a closed subgroup of $G$.
Let $\on{Irr}(\pi_0(Z_G(x))$ denote the set of isomorphism classes of irreducible complex representations of the component group $\pi_0(Z_G(x))$.
Then the enhanced $G$-set of $X$
is by definition the set 
\[X^{G,enh}=\{(x,\chi)| x\in X, \chi\in\on{Irr}(\pi_0(Z_G(x))\}.\]
 We have a natural $G$-action on $X^{G,enh}$
 given by $g\cdot(x,\chi)=(g\cdot x,g\chi g^{-1})$.
We write 
\[\on{Loc}_G(X)=G\backslash X^{G,enh}\]
for the set of $G$-orbits on $X^{enh}$, to be called the set  of irreducible $G$-equivariant local systems on $X$.

\section{Tempiric parameters}\label{Tempiric parameters}
Let $G$ be a complex connected reductive group.
We denote by $G^{alg}\to G$ the pro-algebraic group of the projective limit of finite coverings of $G$.
We have a short exact sequence
\[1\to\pi_1(G)^{alg}\to G^{alg}\to G\to 1\]
where $\pi_1(G)^{alg}$ is the projective limit of 
finite quotients of the fundamental group $\pi_1(G)$ of $G$.

We fix a pinning $(G,B,T,X=\{X_\alpha\}_{\alpha\in\Delta})$ of $G$ and 
a pinned involution $\theta_0\in\on{Aut}(G,B,T,X)$.
Let $\Gamma=\on{Gal}(\bC/\bR)=\{1,\delta\}$ be the 
Galois group of $\bR$ and let 
\[G^\Gamma=G\rtimes_{\theta_0}\Gamma\] be the $L$-group of $\check G$, where $\delta$ acts on $G$ via $\theta_0$.
Let $\rW_\bbR=\bC^\times\sqcup\bC^\times j$, $j^2=-1$ and $jzj^{-1}=\bar z$, be the Weil group of $\bR$.
We have a group extension
\[1\to\bC^\times\to\rW_\bR\to\Gamma\to 1.\]
Let $S^1\subset\bC^\times$ be the unit circle.
Denote by $\rW_\bR^{cpt}=S^1\sqcup S^1j\subset\rW_\bR$ the compact Weil group.
The polar decomposition $\bC^\times\cong S^1\times\bR_{>0}$ induces an isomorphism
\[\rW_\bR\cong\rW_\bR^{cpt}\times\bR_{>0}.\]

We recall the notion of tempiric parameters 
in \cite{AA,ABV}.

\begin{defe}
A tempiric parameter is a Langlands parameter
$\phi:\rW_\bR\to G^\Gamma$ 
whose restriction to the factor 
$\bR_{>0}\subset\rW_\bR\cong\rW_\bR^{cpt}\times\bR_{>0}$ is trivial. 

\end{defe}

Let $\phi:\rW_\bR\to G^\Gamma$ be a tempiric parameter.
Then $\phi$ is completely determined by its restriction to the compact Weil group $\rW_\bR^{cpt}$.
We have $\phi(j)=y\delta$ for some $y\in G$ and the restriction $\phi|_{S^1}:S^1\to G\subset G^\Gamma$
gives rise to a unique cocharacter $\lambda:\mathbb G_m\to G$
such that $\phi(u)=\lambda(u)$, $u\in S^1$.
It follows  that 
the set of tempiric parameters is in bijection with the set consisting of pairs $(\lambda,y)$
where

    \begin{itemize}
	\item [(i)] $\lambda:\bGm\rightarrow G$ is a cocharacter over $\bC$,
	\item [(ii)] $y\in G$ such that $y\theta_0(y)=\lambda(-1)$,
	\item [(iii)] $y\theta_0(\lambda(z)) y^{-1}=\lambda(z^{-1})$, for $z\in\bC^\times$.
    \end{itemize}
For any such pair $(\lambda,y)$ we write $\phi_{\lambda,y}:\rW_\bR\to  G^\Gamma$ the corresponding 
tempiric parameter.

We denote by $\on{Temp}(G,\theta_0)$ the set of tempiric parameters. 
We have a natural conjugation action of $G$ on $\on{Temp}(G,\theta_0)$. If $\phi=\phi_{\lambda,y}$, then the action is given by
 \[g\cdot\phi_{\lambda,y}=\phi_{g\lambda g^{-1},gy\theta_0(g)^{-1}}.\]
Via the projection $G^{alg}\to G$, we get 
a $G^{alg}$-action on $\on{Temp}(G,\theta_0)$.
It is known that, for any $\phi\in\Temp(G,\theta_0)$, $\pi_0(Z_G(\phi))$ is a finite $2$-group and 
 $\pi_0(Z_{G^{alg}}(\phi))$
 is an abelian group (possibly infinite).

\begin{defe}\label{enh}
(1) 
Two tempiric parameters  are called equivalent if they lie in the same $G$-orbit.
We denote by $G\backslash\on{Temp}(G,\theta_0)$ the set of $G$-orbits, referred to as equivalence classes of tempiric parameters.

(2)  Elements in the set $\Temp(G,\theta_0)^{G,enh}=\{(\phi,\chi)|\phi\in\Temp(G,\theta_0), \chi\in\on{Irr}(\pi_0(Z_G(\phi)))\}$
are called $G$-enhanced tempiric parameters.
The set of irreducible $G$-local systems
\[\on{Loc}_G(\Temp(G,\theta_0))=G\backslash\Temp(G,\theta_0)^{G,enh}\]
is referred as equivalence classes of $G$-enhanced tempiric parameters.

(3)  Elements in the set $\Temp(G,\theta_0)^{G^{alg},enh}=\{(\phi,\tilde\chi)|\phi\in\Temp(G,\theta_0), \chi\in\on{Irr}(\pi_0(Z_{G^{alg}}(\phi))\}$
are called $G^{alg}$-enhanced tempiric parameters.
The set of irreducible $G^{alg}$-local systems
\[\on{Loc}_{G^{alg}}(\Temp(G,\theta_0))=G^{alg}\backslash\Temp(G,\theta_0)^{G^{alg},enh}\]
is referred as equivalence classes of $G^{alg}$-enhanced tempiric parameters.
    
\end{defe} 

According to \cite[page 61]{ABV}, the natural map $\pi_0(Z_{G^{alg}}(\phi))\to \pi_0(Z_{G^{}}(\phi))$ 
is surjective and it follows that 
the inflation of a 
character $\chi:\pi_0(Z_{G^{}}(\phi))\to\bC^\times$
along the covering $\pi_0(Z_{G^{alg}}(\phi))
\to\pi_0(Z_{G^{}}(\phi))$
induces a natural embedding 
\begin{equation}\label{inclusion}
    \on{Loc}_{G}(\Temp(G,\theta_0))\hookrightarrow \on{Loc}_{G^{alg}}(\Temp(G,\theta_0)).
\end{equation}

\section{Loop symmetric spaces}\label{loop}
Let $F=\bC((t))$ and $\cO=\bC[[t]]$
be the formal Laurent series and Taylor series rings.
Let $G(F)$ and $G(\cO)$ be the loop group and arc group of $G$. Consider the involution \[\theta:G(F)\to G(F),\ \ \theta(\gamma)(t)=\theta_0(\gamma(-t)).\]
The $\theta$-anti fixed points
\[X^\theta(F)=\{\gamma\in G(F)|\gamma\theta(\gamma)=1\}\]
is called the loop symmetric space of $G(F)$.
The loop group $G(F)$ acts on $X^\theta(F)$ via the 
$\theta$-conjugation $g\cdot\gamma=g\gamma\theta(g)^{-1}$
and we denote by $G(\cO)\backslash X^\theta(F)$ the set of orbits.

In \cite{CY}, we obtain the following parameterization of 
$G(\cO)\backslash X^\theta(F)$. Pick a compact real form 
$\eta_{c,0}:G\to G$ such that $\eta_{c,0}(T)=T$
and $\eta_{c,0}\circ\theta_0=\theta_0\circ\eta_{c,0}=:\eta_0$.
Let $G_c=G^{\eta_{c,0}}$ and $T_c=T^{\eta_{c,0}}$
be the corresponding maximal compact subgroups.
Let $\rW=N_{G_c}(T_c)/T_c$ be the Weyl group.
We fix a representative $w_0\in N_{G_c}(T_c)$
of the longest element in the Weyl group.
Let $X_*(T)^+$ be the set of dominant cocharacters of $T$ with respect to $B$.
For any $\lambda\in X_*(T)^+$, we denote by 
$L_\lambda=Z_G(\lambda)$ the corresponding Levi subgroup of centralizers of $\lambda$.

\begin{defe}
    For any $\lambda\in X_*(T)^+$ satisfying $\lambda=-w_0^{-1}\theta_0(\lambda)$, we let
    \[Y_{\lambda,0}=\{y\in L_\lambda w_0^{-1}|y\theta_0(y)=\lambda(-1)\}.\]
\end{defe}

Note that $\on{Ad}_{w_0^{-1}}\theta_0(L_\lambda)=L_\lambda$ 
since $\lambda=-w_0^{-1}\theta_0(\lambda)$.
Hence the
Levi subgroup $L_\lambda$ acts on $Y_{\lambda,0}$ by $g\cdot y=gy\theta_0(y)^{-1}$ and 
we denote by $L_\lambda\backslash Y_{\lambda,0}$ the set of orbits.

\begin{lem}\label{Y=X}
    The assignment $y\to \gamma_{\lambda,y}=t^\lambda y$
    induces a bijection on obits
    \[
    \bigsqcup_{\lambda=-w_0^{-1}\theta_0(\lambda)\in X_*(T)^+} L_\lambda\backslash Y_{\lambda,0}
    \cong G(\cO)\backslash X^\theta(F).
    \]
    Moreover, the natural inclusion $Z_{L_\lambda}(y)\to Z_{G(\cO)}(\gamma_{\lambda,y})$ of stabilizers 
    induces an isomorphism
    \[\pi_0(Z_{L_\lambda}(y))\cong\pi_0(Z_{G(\cO)}(\gamma_{\lambda,y}))\]
    on component groups.
\end{lem}
\begin{proof}
    Consider the set 
    $A_{\lambda,0}=\{g\in L_\lambda|g=\theta_0(w_0)\theta_0(g^{-1})w_0\lambda (-1)\}$.
    Then $L_\lambda$ acts on $A_{\lambda,0}$ by 
    $h\cdot g=hg\Ad_{w_0^{-1}}\theta_0(h^{-1})$ and a direct computation shows that there is a 
   $L_\lambda$-equivariant bijection 
    \[A_{\lambda,0}\cong L_{\lambda,0}, \ \ g\to y=gw_0^{-1}.\]
Thus we reduce to check that 
the map sending $g\to t^\lambda gw_0^{-1}$
induces a bijection
\[\bigsqcup_{\lambda=-w_0^{-1}\theta_0(\lambda)\in X_*(T)^+} L_\lambda\backslash A_{\lambda,0}\cong G(\cO)\backslash X^\theta(F).\]
The desired claim is proved in \cite[Theorem 13]{CY}.
Indeed, since $\theta_0$ is a pinned involution, we have $\theta_0(B^-)=\Ad_{w_0}(B)$ and hence 
the element $w_1$ in \emph{loc. cit.} can be chosen as $w_1=w_0$.
\end{proof}

Regard $\lambda$ as a cocharacter of $G$.
Note that for any $y=gw_0^{-1}\in Y_{\lambda,0}$ 
we have 
\[y\theta_0(\lambda)y^{-1}=gw_0^{-1}\theta_0(\lambda)w_0g^{-1}=g\lambda^{-1}g^{-1}=\lambda^{-1}.\]
Thus the pair $(\lambda,y)$ satisfies condition $(i)-(iii)$
in Section \ref{Tempiric parameters} and we get a natural $L_\lambda$-equivariant inclusion
\[Y_{\lambda,0}\to\Temp(G,\theta_0),\ \ y\to \phi_{\lambda,y}.\]

\begin{lem}\label{Y=T}
    The assignment $y\to \phi_{\lambda,y}$ above induces a 
    bijection
    \[\bigsqcup_{\lambda=-w_0^{-1}\theta_0(\lambda)\in X_*(T)^+} L_\lambda\backslash Y_{\lambda,0}\cong G\backslash \Temp(G,\theta_0).\]
    Moreover, the inclusion $L_\lambda\subset G$ induces an isomorphism
  \[Z_{L_\lambda}(y)\cong Z_{G}(\phi_{\lambda,y})\]
  between the corresponding stabilizer subgroups.
\end{lem}
\begin{proof}
Let $(\lambda,y)$ be the pair satisfying $(i)-(iii)$ in Section \ref{Tempiric parameters}.
Replacing $\lambda$ by its $G$-conjugate, we can assume 
$\lambda\in X_*(T)^+$. Since both $w_0^{-1}\theta_0(\lambda)$ and $\lambda^{-1}$ are anti-dominant, 
the equation $(iii)$ $y\theta_0(\lambda)y^{-1}=\lambda^{-1}$ implies that 
$w_0^{-1}\theta_0(\lambda)=\lambda^{-1}$ and
$y=gw_0^{-1}$ with $g\in L_\lambda$. Thus the mapping 
$y\to \phi_{\lambda,y}$ induces a surjection on orbits.
The injectivity and the claim on stabilizers follow
 since \[g\cdot\phi_{\lambda,y}=\phi_{g\lambda g^{-1},gy\theta_0(g)^{-1}}=\phi_{\lambda',y'}\]
with $\lambda,\lambda'\in X_*(T)^+$
implies  $\lambda=\lambda'$ and $g\in L_\lambda$.
\end{proof}

Combining Lemma \ref{Y=X} and Lemma \ref{Y=T}, we obtain

\begin{thm}\label{T=X}
    The assignment $\phi_{\lambda,y}\to \gamma_{\lambda,y}=t^\lambda y$
    induces a bijection
    \[ G\backslash\Temp(G,\theta_0)\cong G(\cO)\backslash X^\theta(F).\]
    Moreover, the natural inclusion 
    $Z_G(\phi_{\lambda,y})\to Z_{G(\cO)}(\gamma_{\lambda,y})$ indcues an isomorphism on component groups
    \[\pi_0(Z_G(\phi_{\lambda,y}))\cong \pi_0(Z_{G(\cO)}(\gamma_{\lambda,y})).\]
\end{thm}

Let $G^{alg}(\cO)$ be the arc group of $G^{alg}$ and consider its action on $X^{\theta}(F)$ through the natural map $G^{alg}(\cO)\to G^{}(\cO)$.

\begin{thm}\label{T=X enhanced}
\begin{itemize}
    \item [(1)]
    There is a natural bijection
  \[\on{Loc}_G(\Temp(G,\theta_0))\cong \on{Loc}_{G(\cO)}(X^\theta(F))\]
of irreducible equivariant local systems.
    \item [(2)]
    There is a natural bijection
\[\on{Loc}_{G^{alg}}(\Temp(G,\theta_0))\cong \on{Loc}_{G^{alg}(\cO)}(X^\theta(F))\]
of irreducible equivariant local systems.
\end{itemize}
\end{thm}
\begin{proof}
Part (1) follows from Theorem \ref{T=X}.
For part (2), we need to show that 
the natural map $\pi_0(Z_{G^{alg}}(\phi_{\lambda,y}))\to \pi_0(Z_{G^{alg}(\cO)}(x_{\lambda,y}))$ is an isomorphism.
Let $L_\lambda^{alg}=G^{alg}\times_GL_\lambda$
be the pre-image of $L_\lambda$ along the covering $G^{alg}\to G$.
It follows from \cite[Theorem 13]{CY} that 
the evaluation map $G^{alg}(\cO)\to G^{alg},\ \gamma(t)\to\gamma(0)$
induces a surjective map $Z_{G^{alg}(\cO)}(\gamma_{\lambda,y})\to
Z_{L_\lambda^{alg}}(y)$ with contractible fibers, and thus an isomorphism $\pi_0(Z_{G^{alg}(\cO)}(\gamma_{\lambda,y}))\cong \pi_0(Z_{L_\lambda^{alg}}(y))$.
Thus we reduce to show that the composed map 
$\pi_0(Z_{G^{alg}}(\phi_{\lambda,y}))\to \pi_0(Z_{G^{alg}(\cO)}(\gamma_{\lambda,y}))\to \pi_0(Z_{L_\lambda^{alg}}(y))$ is an isomorphism.
This follows from Lemma \ref{Y=T} as it implies the composed map
$Z_{G^{alg}}(\phi_{\lambda,y})\to Z_{G^{alg}(\cO)}(\gamma_{\lambda,y})\to  Z_{L_\lambda^{alg}}(y)$
is the identity map.
\end{proof}

\section{Twistor $\mathbb P^1$ and Kottwitz sets}\label{Tw}
We recall the Matsuki duality between
$G(\cO)$-orbits on $X^\theta(F)$
and $G$-bundles on twistor $\mathbb P^1$ in \cite{CY}.

Let $\bbP^1$ be the complex projective line with coordinate $t$. Consider the conjugation $c:\bbP^1\to\bbP^1$ sending $t\to -(\bar t)^{-1}$.
Then the twistor-$\bbP^1$, denote by $\widetilde\bbP^1_{\bR}$, is the $\bR$-scheme 
that is obtained by descending $\bbP^1$ to $\bR$ via $\eta$. 
Let $\Bun_G(\bbP^1)$ be the stack of $G$-bundles on $\bbP^1$. Then the conjugations $\eta_0$
and $c$ on $G$ and $\bbP^1$ together give rise to
a conjugation $\eta:\Bun_G(\bbP^1)\to \Bun_G(\bbP^1)$,
and we denote by 
$\Bun_G(\bbP^1)_\eta$ the $\bbR$-stack obtained by descending $\Bun_G(\bbP^1)$ to $\bbR$ via $\eta$, referred to as the stack of $G$-bundles on the twistor-$\bbP^1$.
An $\bbR$-point of $\Bun_G(\bbP^1)_\eta$ consists of 
a $G$-bundle $\cE$ on $\bbP^1$ 
and an isomorphism $c:\cE\cong\eta(\cE)$ 
such that the composition is the identify map
\begin{equation}\label{c^2}
    \cE\stackrel{c}\to\eta(\cE)\stackrel{\eta(c)}\to 
    \eta^2(\cE)=\cE.
\end{equation}
Denote by $\Bun_G(\bbP^1)_\eta(\bbR)$ the set of isomorphism classes of $G$-bundles on the twistor $\bbP^1$
and 
\[\on{Loc}(\Bun_G(\bbP^1)_\eta(\bbR))=\{(\cE,c,\chi)|(\cE,c)\in\Bun_G(\bbP^1)_\eta(\bbR), \chi\in\on{Irr}(\pi_0(\Aut(\cE,c)))\}\]
where $\Aut(\cE,c)$ is the group of automorphism of $(\cE,c)$, referred as the isomorphism classes of local systems on $G$-bundles 
on the twistor $\bbP^1$.

\begin{lem}\label{review of Matsuki}
(1)
There is a natural bijection 
\[G(\cO)\backslash X^\theta(F)\cong \Bun_G(\bbP^1)_\eta(\bbR)\]
between $G(\cO)$-orbits on $X^\theta(F)$
and isomorphism classes of $G$-bundles on twistor-$\bbP^1$. 
Moreover, if the $G(\cO)$-orbits of 
$\gamma\in X^\theta(\cO)$ corresponds to 
$(\cE_\gamma,c_\gamma)\in \Bun_G(\bbP^1)_\eta(\bbR)$ under the bijection, there is an isomorphism 
\[\pi_0(Z_{G(\cO)}(\gamma))\cong\pi_0(\Aut(\cE_\gamma,c_\gamma))\]
between the component group of $G(\cO)$-stabilizer of $\gamma$ and the automorphism group of $(\cE_\gamma,c_\gamma)$.

(2) There is a natural  bijection between local systems
\[\on{Loc}_{G(\cO}(X^\theta(F))\cong\on{Loc}(\Bun_G(\bbP^1)_\eta(\bbR)).\]
\end{lem}
\begin{proof}
 Part (1) is \cite[Theorem 4 and Theorem 22]{CY}. Part (2) is an immediately corollary of (1).
\end{proof}

As a corollary of Theorem \ref{T=X} and Theorem \ref{T=X enhanced}, we obtain 
\begin{thm}\label{twistor=T}
(1)    There is a natural bijection
    \[\Bun_G(\bbP^1)_\eta(\bbR)\cong G\backslash\Temp(G,\eta_0)\]
    between isomorphism classes of $G$-bundles on the twistor $\bbP^1$ and equivalence classes of tempiric parameters.

(2)  There is a natural bijection between local systems
    \[\on{Loc}(\Bun_G(\bbP^1)_\eta(\bbR))\cong \on{Loc}_{G}(\Temp(G,\eta_0)).\]
   
\end{thm}

We make connection with the Kottwitz set.
Denote by 
$B(G,\eta_0)$ 
the Kottwitz set
associated to the real form $\eta_0=\theta_0\circ\eta_{c,0}:G\to G$, see \cite{K}.
By definition, it is the set of $G$-orbits 
$B(G,\eta_0)=G\backslash Z^1_{alg}(\rW_\bbR,G)$
where $Z^1_{alg}(\rW_\bR,G)$ is the set of cocycles $z:\rW_\bR\to G$, $z(ww')=z(w)w(z(w'))$, such that $z|_{\bC^\times}=\lambda(\bC^\times):\bC^\times\to G$
comes from an algebraic cocharacter $\lambda:\mathbb G_m\to G$, and the $G$-action is given by $\eta_0$-conjugation action. 
Concretely, the assignment sending $z
\to (\lambda,y=z(j))$ defines a $G$-equivariant bijection between $Z^1_{alg}(\rW_\bR,G)$ and the $G$-set 
consisting of pairs $(\lambda,y)$ where 
\begin{itemize}
	\item [(i)] $\lambda:\bGm\rightarrow G$ is a cocharacter over $\bC$,
	\item [(ii)] $y\in G$ such that $y\eta_0(y)=\lambda(-1)$,
	\item [(iii)] $y\eta_0(\lambda(z)) y^{-1}=\lambda(\bar z)$, for $z\in\bC^\times$,
	\item [(iv)] $g\cdot (\lambda,y)=(g\lambda g^{-1},gy\eta_0(g)^{-1})$ for $g\in G$.
    \end{itemize}
It is known that there is a natural bijection
\[B(G,\eta_0)\cong\Bun_G(\bbP^1)_{\eta},\]
see for example \cite[Proposition 5.7]{CY}. 
Thus 
Combining with Theorem \ref{twistor=T}, we obtain
\begin{thm}\label{BG=T}
There is a natural
bijection
\[B(G,\eta_0)\cong G\backslash\Temp(G,\theta_0)\]
between the Kottwitz set associated to the real form $\eta_0$
and equivalence classes of tempiric parameters.
\end{thm}

We briefly describe the construction of the bijection in Theorem \ref{BG=T}, which serves as the starting point of this note.
The main observation is that, in \cite[Theorem 17]{CY},
we show that every cocycle $(\lambda,y)$
satisfying
$(i)-(iii)$ above admits a compact representative, that is,
$g\cdot(\lambda,y)=(\lambda',y')$ where $y'\in G_c=G^{\eta_{c,0}}$. 
Note that $\theta_0(y')=\eta_0\eta_{c,0}(y')=\eta_0(y')$. Thus by applying the compact conjugation $\eta_{c,0}$ to 
equations $(ii)-(iii)$ we get 
\[y'\theta_0(y')=\lambda'(-1),\]
\[y'\theta_0(\lambda'(z))(y')^{-1}=\lambda'(z^{-1}),\]
and  the pair $(\lambda',y')$ defines the corresponding temperic parameter 
$\phi_{\lambda',y'}$.

\section{Minimal $K$-types}
Let $\check G$ be the complex dual group. 
We fix a pinning $(\check G,\check B,\check T,\check X)$ and let $\check\theta_0\in\on{Aut}(\check G,\check B,\check T,\check X)$ be the pinned involution 
corresponding to $\theta_0$ under the 
natural isomorphism $\on{Aut}(\check G,\check B,\check T,\check X)\cong \on{Aut}(G,B,T,X)$.
Let $\on{Chev}_{\check G}\in \on{Aut}(\check G,\check B,\check T,\check X)$ be the Chevalley involution such that $\on{Chev}_{\check G}(t)=w_0(t^{-1})$ for $t\in\check T$.
The composition 
\[
\check\theta=\on{Chev}_{\check G}\circ\check\theta_0
\]
is a pinned involution of $\check G$, to be called the
dual of $\theta_0$.
Let $Z(\check G)^{\check\theta}$ and $Z(\check G)_{tor}$
be the subgroup of $\check\theta$-fixed points and torsion elements in $Z(\check G)$.
We set $Z(\check G)^{\check\theta}_{tor}=Z(G)^{\check\theta}\cap Z(G)_{tor}$.
Note that $2\check\rho(-1)\in Z(\check G)^{\check\theta}_{tor}$, where $2\check\rho$ is the sum of 
positive coroots.
Recall the notions of pure involutions and strong involutions in \cite{ABV}:

\begin{defe}
    A strong involution in the inner class of $\check\theta$
    is an element $(c,\check\theta)\in \check G\rtimes\{1,\check\theta\}$ satisfying
    $z=c\check\theta(c)\in Z(\check G)^{\check\theta}_{tor}$.
    Denote $\check\theta_c=\Ad_c\circ\check\theta$. 
    The central element $z$ is called the central invariant of $(c,\check\theta)$. 
    Two strong involutions $(c,\check\theta),(c',\check\theta)$
    are called equivalent if 
    $c=gc'\check\theta(g)^{-1}$ for some $g\in\check G$.

    Write $I$ for the set of equivalent classes of strong involutions in the inner class of $\check\theta$.
    For any $z\in Z(\check G)^{\check\theta}_{tor}$, we write  $I_z\subset I$ for the subset of equivalent classes of strong involutions with central $z$.
    In particular, we call elements of $I_{2\rho(-1)}$
    the pure involutions.
\end{defe}

We fix a set of  representatives $\check\theta_i, i\in I$ and let $\check K_i=\check G^{\check\theta_i}$
be the associated symmetric subgroups. 
Vogan's theory of minimal $K$-types \cite{V} establishes remarkable bijections between the following sets:

\begin{itemize}
	\item [(i)] The 
    set $\on{Loc}_{G^{alg}}(\Temp(G,\theta_0))$
    of equivalence classes of
    $G^{alg}$-enhanced tempiric parameters, see Definition \ref{enh},
	\item [(ii)] The union, over $i\in I$, of  
   the tempered, irreducible
  $(\check\fg,\check K_i)$-modules with real infinitesimal characters,
	\item [(iii)] The union, over $i\in I$, of the 
    irreducible representations $\on{Irr}(\check K_i)$
    of $\check K_i$.
    \end{itemize}
The map $(\phi,\tilde\chi)\rightarrow\pi(\phi,\tilde\chi)$
from
$(i)$ to $(ii)$ is the restriction of  the local Langlands correspondence. Then map 
$\pi(\phi,\tilde\chi)\rightarrow \on{LKT}(\pi(\phi,\tilde\chi))$ from $(ii)$ to $(iii)$
is given by its unique minimal $K$-type.

We review how to determine the strong involution $\check{K}_i$
for a given $G^{alg}$-enhanced tempiric parameter $(\phi,\tilde\chi)$ following \cite{AA}.
By \cite[Lemma 3.2]{AA}, by $G$-conjugation, we can assume
$\phi(\bC^\times)\subset T$ and $\phi(j)$ normalizes $T$.
Moreover, we can choose the conjugation so that $\phi$
is in the \emph{standard form}\cite[Definition 4.1]{AA}.
The conjugation by $\phi(j)$ defines an involution $\tau$ on $T$. 
Let $T^{alg}\subset G^{alg}$ be the preimage of $T\subset G$
and $T^{alg,\tau}$ the preimage of $T^\tau$.
Denote by $\check\tau=\check\tau(\phi)$ the composition of the inversion with the dual involution of $\tau$ on $\check T$.

Recall that the \emph{KGB space} for the inner class of $\check\theta$ is
\begin{equation}
    \mathcal{X}=\{(c,\check\theta)\in\on{Norm}_{\check G\rtimes\{1,\check\theta\}\backslash\check G}(\check T)\mid c\check\theta(c)\in Z(\check G)\}/\check T,
    \end{equation}
where $\check T$ acts by conjugation.
For $\check\tau$ on $\check T$, denote
$\mathcal{X}_{\check\tau}=\{\xi=(c,\check\theta)\in\mathcal{X}\mid\Ad_\xi|_{\check T}=\check\tau\}$.
For $z\in Z(\check{G})$, denote
$\mathcal{X}(z)=\{(c,\check\theta)\in\mathcal{X}\mid c\check\theta(c)=z\}$.
By \cite[Proposition 4.12]{AA}, we have a bijection
$\on{Irr}(\pi_0(T^{alg,\tau}))\simeq\mathcal{X}_{\check\tau}$.
As mentioned in \cite[\S4.4]{AA}, 
there are natural surjections
$p_\phi^0:\pi_0(T^\tau)\rightarrow \pi_0(Z_G(\phi)),\ 
p_\phi:\pi_0(T^{alg,\tau})\rightarrow \pi_0(Z_{G^{alg}}(\phi))$,
inducing embeddings of the sets of characters.
Together we obtain an embedding
\begin{equation}\label{eq:KGB fiber}
    \mathcal{E}_\phi:\on{Irr}(\pi_0(Z_{G^{alg}}(\phi)))\hookrightarrow\mathcal{X}_{\check\tau}.
\end{equation}
A representative of $\mathcal{E}_\phi(\tilde\chi)$
gives the corresponding (equivalence class of) 
strong involution for $(\phi,\tilde\chi)$.

Recall the natural inclusion 
$\on{Loc}_{G}(\Temp(G,\theta_0))\hookrightarrow \on{Loc}_{G^{alg}}(\Temp(G,\theta_0))$
in~\eqref{inclusion}.
   
\begin{lem}\label{pure}
   The bijection between $(i)$ and $(iii)$
   restricts to a bijection 
   \[\on{Loc}_{G}(\Temp(G,\theta_0))\cong\bigsqcup_{i\in I_{2\check\rho(-1)}}\on{Irr}(\check K_i), \ \ (\phi,\chi)\to\on{LKT}(\pi(\phi,\chi))\]
 between the equivalent classes of $G$-enhanced tempiric parameters and the union, over the set $I_{2\check\rho(-1)}$ of
 pure involutions, 
 of the irreducible representations of $\check K_i$.
\end{lem}
\begin{proof}
    For a $G$-enhanced tempiric parameters $(\phi,\chi)$,
    we lift $\chi$ to $\tilde\chi$ via the surjective map
    $\pi_0(Z_{G^{alg}}(\phi))\rightarrow\pi_0(Z_G(\phi))$.
    By \cite[Proposition 4.17]{AA},
    $\mathcal{E}_\phi(\tilde\chi)\in\mathcal{X}_{2\check\rho(-1)}$,
    so that the corresponding strong involution belongs to $I_{2\check\rho(-1)}$.

    It remains to show that the restriction of Vogan's bijection
    is surjective.
    It suffices to show that for a 
    $G^{alg}$-enhanced tempiric parameter $(\phi,\tilde\chi)$,
    if $\mathcal{E}_\phi(\tilde\chi)\in\mathcal{X}_{2\check\rho(-1)}$,
    then $\tilde\chi$ factors through 
    $\pi_0(Z_{G^{alg}}(\phi))\rightarrow\pi_0(Z_G(\phi))$.
    Consider commutative diagram
    \[\begin{tikzcd}
    {\pi_0(T^{alg,\tau})} \arrow[r, "p_\phi"] \arrow[d, "q_\tau"'] & \pi_0(Z_{G^{alg}}(\phi)) \arrow[d, "q_\phi"'] \\
    \pi_0(T^\tau) \arrow[r, "p_\phi^0"]                            & \pi_0(Z_G(\phi))                             
    \end{tikzcd}\]
   where all morphisms are natural surjections.
   From the construction of 
   $\on{Irr}(\pi_0(T^{alg,\tau}))\simeq\mathcal{X}_{\check\tau}$
   in \cite[Lemma 4.10, \S4.3]{AA},
   we have
   $\on{Irr}(\pi_0(T^\tau))\simeq\mathcal{X}_{\check\tau}(2\check\rho(-1))$.
   The problem is reduced to showing that if $\tilde\chi$
   factors through $\pi_0(Z_{G^{alg}}(\phi))$
   and $\pi_0(T^\tau)$,
   then it factors through $\pi_0(Z_G(\phi))$.
   To this end, it suffices to show that
   \begin{equation}\label{eq:ker comparison}
   q_\tau(\ker(p_\phi))=\ker(p_\phi^0).
   \end{equation}

   As explained in \cite[\S4.4]{AA}, the kernel of $p_\phi$
   is generated by the following elements.
   Assume $\phi|_{\bC^\times}=\lambda\in X_*(T)$.
   Consider $\tau$-real and $\phi$-sinigular roots for $(G,T)$,
   i.e. those roots $\alpha$ such that
   $\tau(\alpha)=-\alpha$ and $(\alpha,\lambda)=0$.
   Consider the associated morphism 
   $\mathrm{SL}_2(\bC)\rightarrow G^{alg}$
   that lifts the morphism $\SL_2\rightarrow G$
   associated to $\check{\alpha}$.
   The image of $\check\alpha(-I)$ lands inside $T^{alg,\tau}$
   and maps to an element $\overline{m}_\alpha\in\pi_0(T^{alg,\tau})$.
   Then $\ker(p_\phi)$ is generated by these $\overline{m}_\alpha\in\pi_0(T^{alg,\tau})$.

   On the other hand, 
   $L_\lambda=Z_G(\phi(\bC^\times))$ is a Levi subgroup, 
   on which $\tau$ acts.
   In \cite[Lemma 12.10]{ABV},
   let their $(G,\theta,K,T)$ be our 
   $(L_\lambda,\tau, L_\lambda^\tau=Z_G(\phi),T)$.
   Since $\phi$ is chosen to be of standard form,
   the assumption (a) of the lemma is satisfied. 
   Note $T\cap K=T^\tau$.
   Thus by (d) of the lemma, the kernel of
   $p_\phi^0:\pi_0(T^\tau)=\pi_0(T\cap K)\rightarrow\pi_0(Z_G(\phi))=\pi_0(K)$
   is generated by the images $m_\alpha'\in\pi_0(T^\tau)$
   of $\check{\alpha}(-1)\in T^\tau$
   for $\tau$-real roots $\alpha$ in $L_\lambda$,
   i.e. $\tau$-real $\phi$-singular roots of $G$.
   Clearly $q_\tau(\overline{m}_\alpha)=m_\alpha'$.
   This proves \eqref{eq:ker comparison}.
\end{proof}

Combining with the bijection 
in Theorem \ref{T=X enhanced}, we obtain:

\begin{thm}\label{K types}
    (1) The assignment $(\gamma,\tilde\chi)\to\on{LKT}(\pi(\phi,\tilde\chi))$
    induces a bijection
    \[\on{Loc}_{G^{alg}(\cO)}(X^\theta(F))\cong\bigsqcup_{i\in I}\ \on{Irr}(\check K_i).\]

       (2) The assignment $(\gamma,\chi)\to\on{LKT}(\pi(\phi,\chi))$
    induces a bijection
    \[\on{Loc}_{G^{}(\cO)}(X^\theta(F))\cong\bigsqcup_{i\in I_{2\check\rho(-1)}}\on{Irr}(\check K_i).\]
\end{thm}

\section{Generalizations}\label{generalizations}
We state a generalization to 
extended loop symmetric spaces, generalized 
$G$-bundles, 
extended Kottwitz sets, tempiric $E$-parameters, and 
projective $(\check\fg,\check K_i)$-modules.

Let $Z(G)^{\theta_0}_{tor}$
be the torsion elements in the $\theta_0$-fixed points subgroup of $Z(G)$.
We fix an element $z\in Z(G)^{\theta_0}_{tor}$.
Following \cite[Section 5.4.2]{CY},
consider the extended loop symmetric space 
\[X_z^\theta(F)=\{\gamma\in G(F)|\gamma\theta(\gamma)=z^{-1}\},\]
the stack of generalized $G$-bundles on twistor $\bbP^1$
\[\Bun_G(\bbP^1)_z=\{(\cE,c)|\cE\in\Bun_G(\bbP^1), c:\cE\cong\eta(\cE), c\eta(c)=z^{-1}\on{id}_\cE:\cE\to\eta(\cE)\to \eta^2(\cE)=\cE\}\]
and the extended Kottwitz set 
\[B(G,\eta_0)_z=G\backslash Z^1_{alg}(\rW_\bR,G,z)\]
where $Z^1_{alg}(\rW_\bR,G,z)$
is the set consisting of pairs $(\lambda,y)$ where 
\begin{itemize}
	\item [(i)] $\lambda:\bGm\rightarrow G$ is a cocharacter over $\bC$,
	\item [(ii)] $y\in G$ such that $y\eta_0(y)=\lambda(-1)z^{-1}$,
	\item [(iii)] $y\eta_0(\lambda(t)) y^{-1}=\lambda(\bar t)$, for $t\in\bC^\times$,
	\item [(iv)] $g\cdot (\lambda,y)=(g\lambda g^{-1},gy\eta_0(g)^{-1})$ for $g\in G$.
    \end{itemize}
Let $\Bun_G(\bbP^1)_{z}(\bbR)$ be the set of isomorphism classes of generalized $G$-bundles on twistor $\bbP^1$
and let $\on{Loc}(\Bun_G(\bbP^1)_{z}(\bbR))$ be the set of 
isomorphism classes of local systems on it.

We have the following generalization of Matsuki duality
in \cite[Section 5.4]{CY}:

\begin{lem}\label{extended}
(1) There are natural bijections
     \[G(\cO)\backslash X^{\theta }_z(F)\cong \Bun_G(\bbP^1)_{z}(\bbR)\cong B(G,\eta_0)_z. \]
(2) There is a natural bijection
\[\on{Loc}_{G^{}(\cO)}( X^{\theta }_z(F))\cong 
    \on{Loc}(\Bun_G(\bbP^1)_{z}(\bbR)).\]
\end{lem}

To define tempiric $E$-parameters, we first 
recall the notion of $E$-group $G^\Gamma_z$
with second invariant $z$
in \cite{ABV}.
As a set, we have 
$ G^\Gamma_z=G\sqcup G\delta$,
and the multiplication is given by 
\[\delta g\delta ^{-1}=\theta_0(g),\ \  \delta^2=z.\]
Recall $\Gamma=\on{Gal}(\bC/\bR)$. 
We have a short exact sequence 
\[1\to G\to G^\Gamma_z\to \Gamma\to 1.\]
When $z=1$, then $G_1^\Gamma\cong G^\Gamma= G\rtimes_{\theta_0}\Gamma$ is the $L$-group. In general, the $E$-group $G^\Gamma_z$
is not isomorphism to a semi-direct product of $G$ and $\Gamma$. A tempiric $E$-parameter is a 
continuous map $\phi:\rW_\bR\to G^\Gamma_z$ over $\Gamma$
whose restriction to the factor $\bR_{>0}$ is trivial.
We denote by $\Temp(G,\theta_0)_z$ the set of 
tempiric $E$-parameters and 
$G\backslash\Temp(G,\theta_0)_z$ the set of $G$-orbits under the natural conjugation action.

Consider the pro-algebric group 
$\check G^{alg}\to\check G$ of projective limit of finite coverings of $\check G$.
For each $\check K_i, i\in I$, we 
write $\check K^{alg}=\check K_i\times_{\check G}\check G^{alg}$
the base change.
According to \cite[Section 10]{ABV},
there is an isomorphism
\[\on{Hom}_{cont}(\pi_1(\check G)^{alg},\bC^\times)\cong Z(G)_{tor}.\]
Hence the element $z\in Z(G)_{tor}^{\theta_0}$ gives rise to a continuous character 
$\chi_z:\pi_1(\check G)^{alg}\to\bC^\times$
of $\check G^{alg}$.
Let $\on{Irr}(\check K_{i}^{alg})$ be the set of  irreducible finite dimensional 
representations of $\check K^{alg}$ and  
let  
\[\on{Irr}(\check K^{alg}_{i})_z\subset \on{Irr}(\check K_{i}^{alg})\] be the subset of 
those irreducible representations such that  
central subgroup $\pi_1(\check G)^{alg}\subset\check K^{alg}_{i}$
acts by the character $\chi_z$.

\begin{lem}\label{proj}
(1) There is a natural bijection
\[\Loc_{G^{alg}}(\Temp(G,\theta_0)_z)\cong\bigsqcup_{i\in I}\on{Irr}(\check K^{alg}_i)_z.\]
(2) The bijection above restricts to a bijection
\[\Loc_{G^{}}(\Temp(G,\theta_0)_z)\cong\bigsqcup_{i\in I_{2\check\rho(-1)}}\on{Irr}(\check K^{alg}_i)_z.\]
    \end{lem}
\begin{proof}
 It follows 
from
\cite[Theorem 10.4]{ABV} that  there is a bijection between the equivalence classes of tempiric $E$-parameters 
$\Loc_{G^{alg}}(\Temp(G,\theta_0)_z)$ and the isomorphism classes of irreducible tempered $(\check\fg,\check K_i^{alg})$-modules with real infinitesimal characters, whose  action of $\pi_1(\check G)^{alg}$ factors through $\chi_z$.
Note that the character $\chi_z$ factors through 
a finite quotient $\pi_1(\check G)^{alg}\to A_z$, thus 
the action of $\check K_i^{alg}$ also factors through a finite central covering $\check K_{i,z}\to\check K_i$.
Thus we are in the ordinary Harish–Chandra module setting (no pro-algebraic issue)
and we can apply Vogan's theory \cite{V} to conclude the bijection in (1).

Since \cite{ABV} works entirely in the setting of $E$-groups, the argument of Lemma \ref{pure} applies without change. This proves part (2).
\end{proof}

We now can state:

\begin{thm}\label{Main:gen}

     (1) There are natural bijections
    \[G(\cO)\backslash X^{\theta }_z(F)\cong \Bun_G(\bbP^1)_{z}(\bbR)\cong B(G,\eta_0)_z \cong G\backslash\Temp(G,\theta_0)_z.\]

  (2) There are natural bijections
\[\on{Loc}_{G^{}(\cO)}( X^{\theta }_z(F))\cong
\on{Loc}(\Bun_G(\bbP^1)_{z}(\bbR))\cong
\Loc_{G^{}}(\Temp(G,\theta_0)_z)\cong\bigsqcup_{i\in I_{2\check\rho(-1)}}\on{Irr}(\check K^{alg}_i)_z.\]
    
(3) There are natural bijections
\[\on{Loc}_{G^{alg}(\cO)}(X^{\theta }_z(F))\cong\Loc_{G^{alg}}(\Temp(G,\theta_0)_z)\cong\bigsqcup_{i\in I}\on{Irr}(\check K^{alg}_i)_z.\]

\end{thm}
\begin{proof}
For any $\lambda\in X_*(T)^+$ satisfying $\lambda=-w_0^{-1}\theta_0(\lambda)$, we let
    \[Y_{\lambda,0,z}=\{y\in L_\lambda w_0^{-1}|y\theta_0(y)=\lambda(-1)z^{-1}\}.\]
    By the same arguments, all the results of Section \ref{loop} remain valid upon replacing $Y_{\lambda,0}$, $X^\theta(F)$, and $\Temp(G,\theta_0)$ by $Y_{\lambda,0,z}$, $X^\theta_z(F)$, and $\Temp(G,\theta_0)_z$, respectively. We thus obtain bijections
    \[G(\cO)\backslash X^{\theta }_z(F)\cong G\backslash\Temp(G,\theta_0)_z, \]
    \[\on{Loc}_{G(\cO)}(X^{\theta }_z(F))\cong \Loc_{G^{}}(\Temp(G,\theta_0)_z),\]
   \[\on{Loc}_{G^{alg}(\cO)}(X^{\theta }_z(F))\cong \Loc_{G^{alg}}(\Temp(G,\theta_0)_z).\]
Combining with Lemma \ref{extended}
and Lemma \ref{proj}, we obtain the desired bijections.

\end{proof}

\section{Conjectures}\label{conj}
We retain the notation in Section \ref{generalizations}.
Let $\on{Rep}(\check K^{alg}_i)_z$
be the category of finite dimensional complex representations of $\check K^{alg}$  
such that  
central subgroup $\pi_1(\check G)^{alg}\subset\check K_{i}^{alg}$
acts by the character $\chi_z$.
Note that $\on{Rep}(\check K^{alg}_i)_z$ is natually a module category over $\on{Rep}(\check G)$
where the action is through the pull-back functor $\on{Rep}(\check G)\to\on{Rep}(\check K^{alg}_i)$
along the map $\check K^{alg}_i\to\check G$.

It follows from the placidness results \cite{CY2} that 
there is a well-defined abelian category 
$\on{Perv}(G^{alg}(\cO)\backslash X^\theta_z(F))$ of $G^{alg}(\cO)$-equivariant perverse sheaves on $X^\theta_z(F)$. Moreover, the irreducible objects are indexed by 
$\on{Loc}_{G^{alg}(\cO)}(X^\theta_z(F))$.

\begin{conj}\label{Conjectures}

(1) There is an equivalence of abelian  categories
\[\mathrm{Perv}(G^{alg}(\cO)\backslash X_z
^\theta(F))\cong\bigoplus_{i\in I}\on{Rep}(\check K^{alg}_i)_z\]
such that the induced map between irreducible object is given by the bijection in Theorem \ref{Main:gen}.
Moreover,
the equivalence intertwines the 
Hecke action of 
$\on{Perv}(G(\cO)\backslash G(F)/G(\cO))$ on $\mathrm{Perv}(G^{alg}(\cO)\backslash X_z
^\theta(F))$ with the action of $\on{Rep}(\check G)$
on $\on{Rep}(\check K^{alg})_z$
under the geometric Satake equivalence.

(2) The equivalence in (1)  restricts to an equivalence of abelain 
$\on{Rep}(\check G)$-module categories
\[\mathrm{Perv}(G^{}(\cO)\backslash X_z
^\theta(F))\cong\bigoplus_{i\in I_{2\check\rho(-1)}}\on{Rep}(\check K^{alg}_i)_z.\]
\end{conj}

\begin{rem}
    Conjecture \ref{Conjectures} provides an abelian refinement of 
    the Twisted Relative Langlands Duality 
    in \cite[Conjecture 1.5]{C}. 

%(2) With some modifications, 
%we also expect a version of Conjecture \ref{Conjectures} in the modular coefficient settings.
\end{rem}

\section{$\SL_2$ Example}
For $G=\SL_2$, we compute
$\on{Loc}_{G(\cO)}(X^\theta(F))$,
$\on{Loc}_G(\Temp(G,\theta_0))$,
and $\bigsqcup_{i\in I}\on{Irr}(\check K_i)$
explicitly
and verify Theorem \ref{T=X enhanced} and Theorem \ref{K types}.
Note that since $\SL_2$ is simply connected
and $\LG=\PGL_2$ is adjoint,
$G=G^{alg}$
and $I=I_{2\rho(-1)}$.
The only pinned involution on $\SL_2$ is $\theta_0=1$.
Thus we have $\theta:\SL_2(F)\to\SL_2(F), \theta(\gamma)(t)=\gamma(-t)$
and $X^\theta(F)\cong \SL_2(F)/\SL_2(F')$
where $F'=\bC((t^2))$, see \cite[Proposition 25]{CY}.

Choose $T\subset B\subset\SL_2$ to be the diagonal and upper triangular matrices,
and let $\alpha\in X^*(T)$, $\check\alpha\in X_*(T)$
be the unique root and coroot.
Then $X_*(T)^+=\bZ_{+}\check\alpha$.
Let $\check T\subset\PGL_2$ be the diagonal matrices.

\subsection{$\Loc_{G(\cO)}(X^\theta(F))$ and $\Loc_G(\Temp(G,\theta_0))$}
By Lemma \ref{Y=X} and Lemma \ref{Y=T},
it suffices to compute
$\bigsqcup_{\lambda=-w_0^{-1}\theta_0(\lambda)\in X_*(T)^+} L_\lambda\backslash Y_{\lambda,0}$
and $\pi_0(Z_{L_\lambda}(y))$,
where $L_\lambda$ acts on $Y_{\lambda,0}$ by conjugation.
Here $w_0$ acts as $-1$ on $X_*(T)$,
so that every $\lambda\in X_*(T)^+$ satisfies 
$\lambda=-w_0^{-1}\theta_0(\lambda)$.
We fix $w_0=\begin{pmatrix}
    0&-1\\
    1&0
\end{pmatrix}$.

For $\lambda=0$, $L_\lambda=G$,
$Y_{0,0}=\{y\in G\mid y^2=1\}$.
We get $G\backslash Y_{0,0}=\{I,-I\}$.
The stabilizers are both $G=\SL_2$, which is connected.

For $\lambda=m\check\alpha$, $m>0$, $L_\lambda=T$.
Then
\[
Y_{\lambda,0}=\{y=\begin{pmatrix}
    0&a\\
    -a^{-1}&0
\end{pmatrix},a\in\bC^\times\mid y^2=-I=(-1)^m\}.
\]
Thus $m=2r-1$ must be odd, $r>0$. 
The action of $T$ on $Y_{\lambda,0}$ is transitive
with an orbit representative $y=w_0$.
The group of stabilizers is $\{\pm I\}$.
Denote its two characters by $\chi_+,\chi_-$,
$\chi_+(-I)=1,\chi_-(-I)=-1$.

We obtain
\begin{align*}
    \Loc_{G(\cO)}(X^\theta(F))&=\{(I,1),(-I,1),(t^{(2r-1)\check\alpha}w_0,\chi_\pm),r\geq 1\}\\
    \simeq\Loc_G(\Temp(G,\theta_0))&=
    \{(\phi_{0,I},1),(\phi_{0,-I},1),
    (\phi_{(2r-1)\check\alpha,w_0},\chi_\pm),r\geq 1\},
\end{align*}
where elements in the two sets are corresponded in the presented order.

\subsection{Inner forms}
\subsubsection{}
On $\LG=\PGL_2$, the pinned Chevalley involution 
$\on{Chev}_{\check G}$ is trivial,
so that $\check\theta=1$.
Thus $I$ is just the set of involution conjugacy classes in $\PGL_2$, represented by 
$I,\rho(-1)=\begin{pmatrix}
    1&0\\
    0&-1
\end{pmatrix}$.
The fixed subgroups by their adjoint actions are
$\check K_1=\PGL_2,\check K_2=\on{PO}_2=N_{\PGL_2}(\check T)$.

For $\check K_1=\PGL_2=\SL_2/\{\pm I\}$, 
its irreducible representations are those of $\SL_2$
on which $-I$ acts trivially, i.e. the even highest weight representations $V_{2r-2}=\Sym^{2r-2}\bC^2$, $r\geq1$.
The associated real form is the compact real form 
$\on{PU}_2$, whose maximal compact subgroup is itself.
Under Vogan norm, the minimal $K$-type $V_{2r-2}$ of 
a $\on{PU}_2$-representation has minimal $r$.

For $\check K_2=\on{PO}_2$, 
$\on{PO}_2\simeq\check T\rtimes\mu_2\simeq\bGm\rtimes\mu_2$,
where $-1\in\mu_2$ acts on $\bGm$ by inversion.
Denote the weight $m$ character of $\bGm$ by $\chi_m$, $m\in\bZ$.
Then $\Irr(\on{PO}_2)=\{W_r,r\geq 0\}$, where
\[
W_r=\begin{cases}
    \chi_0,\ -1\in\mu\text{ acts trivially},\hspace{5cm} r=0,\\
    \chi_0,\ -1\in\mu\text{ acts as }-1,\hspace{5.2cm} r=1,\\
    \chi_{r-1}\oplus\chi_{-(r-1)},\ -1\in\mu_2\text{ exchanges }\chi_{r-1},\chi_{-(r-1)},\qquad r\geq 2.
\end{cases}
\]
The associated real form is the split real form
$\PGL_2(\bR)$, whose maximal compact subgroup associated to $\on{PO}_2$ is $\on{PO}_2(\bR)$.
The Vogan norm of $W_r$ is $0$ for $r=0,1$ and a positive constant times $r-1$ for $r\geq 2$.

\subsubsection{}
We explain which inner form is associated to 
a given enhanced tempiric parameter $(\phi,\chi)$ for $\SL_2$.

If the character $\chi$ of $\pi_0(Z_G(\phi))$ is trivial,
by \cite[Proposition 4.12]{AA},
the associated point in the KGB fiber is the base point $x_{b,\tau}$, which is conjugated to the distinguished base point
$\rho(-1)$.
The fixed subgroup is $\on{PO}_2$.
Thus $(\phi_{0,I},1),(\phi_{1,I},1),(\phi_{(2r-1)\check\alpha,w_0},\chi_+),r\geq 1$
all correspond to $\on{PO}_2$,
i.e. the split real form $\PGL_2(\bR)$.

If the character $\chi$ is nontrivial, 
the enhanced parameters are $(\phi_{(2r-1)\check\alpha,w_0},\chi_-),r\geq 1$.
We determine the image of $\chi_-$ in the KGB fiber
following \cite[\S4.3]{AA}.
Since $\phi(j)=w_0$ acts on $T$ by inversion,
$\tau=1$ on $\check T$.
The KGB fiber is
\[
\mathcal{X}_\tau
=\{\xi\in N_{\PGL_2}(\check T)\mid \xi^2=1,\Ad_\xi\mid_{\check T}=id\}/_{\Ad}\check{T}
=\{I,\rho(-1)\}.
\]
Under the embedding $\Irr(\pi_0(Z_G(\phi)))\hookrightarrow\mathcal{X}_\tau$,
since the trivial character $\chi_+$ goes to the base point $\rho(-1)$,
the nontrivial character $\chi_-$ must go to the other involution $I$, which fixes $\PGL_2$.
This corresponds to the compact real form $\on{PU}_2$.

\subsection{Tempiric representations and minimal $K$-types}
\subsubsection{Trivial component character, principal series}
For $(\phi_{0,I},1)$ and $(\phi_{0,-I},1)$, 
the associated tempiric representations of $\PGL_2(\bR)$
can be constructed via principal series as follow.
Let $M=\check T[2]=\{I,\rho(-1)\}$, $\check N$ the strictly upper triangular matrices, $A=\{\on{diag}(\bR_{>0},1)\}$.
The irreducible $\PGL_2(\bR)$-representation
associated to $(\phi_{0,I},1),(\phi_{0,I},1)$ are
\[
\pi_{s,0}:=\on{Ind}^{\PGL_2(\bR)}_{MAN}\bC,\quad
\pi_{s,1}:=\on{Ind}^{\PGL_2(\bR)}_{MAN}sgn\otimes\bC
\]
respectively. Here $sgn$ is the nontrivial character of $M$.

When restricted to $\on{PO}_2(\bR)$, by Frobenius reciprocity, 
since $\on{Hom}_M(W_r\!\!\mid_M,W)\neq0$ for $W=\bC$ or $sgn$,
$W_0$(resp. $W_1$) appears in $\pi_{s,0}$(resp. $\pi_{s,1}$). 
Thus the minimal $K$-types of $\pi_{s,0},\pi_{s,1}$
are $W_0,W_1$ respectively.

\subsubsection{Trivial component character, discrete series}
For $(\phi_{(2r-1)\check\alpha,w_0},\chi_+),r\geq 1$,
the Weil representation is irreducible.
The associated representation of $\PGL_2(\bR)$ is as follows.
Let $D_{2r}^{\pm}$ be the holomorphic and anti-holomorphic 
discrete series of $\SL_2(\bR)$
of weight $2r$.
The action of $\SL_2(\bR)$ factors through $\on{PSL}_2(\bR)$.
Note $\PGL_2(\bR)/\on{PSL}_2(\bR)\simeq\mu_2$.
Then 
\[
\pi_{s,r+1}:=\Ind^{\PGL_2(\bR)}_{\on{PSL}_2(\bR)}D_{2r}^+
\]
is the associated representation.

To see its minimal $K$-type, observe
$\pi_{r+1}|_{\on{PSL}_2(\bR)}\simeq D_{2r}^+\oplus D_{2r}^-$.
Restrict it to $\on{PO}_2(\bR)^\circ\simeq\on{PSO}_2(\bR)$.
The $\on{SO}_2(\bR)$-weights of $D_{2r}^+$ and $D_{2r}^-$ 
are $2r, 2r+2, 2r+4...$ and $-2r, -2r-2, -2r-4,...$ respectively,
and $\rho(-1)\in\on{PO}_2(\bR)$ exchanges weights $2r+2a,-2r-2a$.
The minimal weight for the double cover torus is $r$.
Thus the minimal $K$-type of $\pi_{s,r+1}$ is $W_{r+1}$.

\subsubsection{Nontrivial component character, discrete series}
For $(\phi_{(2r-1)\check\alpha,w_0},\chi_-),r\geq 1$,
the Weil representation is irreducible.
The associated irreducible representation of $\on{PU}_2$ is one of $V_{2m-2}$.
Since the infinitesimal character is $2r-1$ for the enhanced parameter and is $2m-1$ for $V_{2m-2}$,
to have them matched we obtain $m=r$.
The associated representation is $V_{2r-2}$,
which is also its minimal $K$-type.

\subsection{The bijection}
We conclude that for $G=\SL_2$, 
the bijections in Theorem \ref{T=X enhanced} 
and Theorem \ref{K types} are given by:
\begin{align*}
     \Loc_{G(\cO)}(X^\theta(F))\xrightarrow{\sim}
     \Loc_{\SL_2}(\Temp(\SL_2,id))&\xrightarrow[\on{LKT}]{\sim}\Irr(\PGL_2)\sqcup\Irr(\on{PO}_2)\\
    (I,1)\mapsto\hspace{1.3cm}(\phi_{0,I},1)\hspace{1.3cm}&\mapsto\qquad W_0\\
    (-I,1)\mapsto\hspace{1.18cm}(\phi_{0,-I},1)\hspace{1.2cm}&\mapsto\qquad W_1\\
    (t^{(2r-1)\check\alpha}w_0,\chi_+)\mapsto\hspace{0.55cm}(\phi_{(2r-1)\check\alpha,w_0},\chi_+)\hspace{0.6cm}&\mapsto \qquad W_{r+1}\qquad r\geq 1\\
    (t^{(2r-1)\check\alpha}w_0,\chi_-)\mapsto\hspace{0.55cm}(\phi_{(2r-1)\check\alpha,w_0},\chi_-)\hspace{0.6cm}&\mapsto\qquad V_{2r-2}\qquad r\geq 1.
\end{align*}


\begin{bibdiv}
\begin{biblist}

\bib{AA}{article}
{
	title={Lowest $K$-types in the local Langlands correspondence}, 
	author={Adams, J},
    author={Afgoustidis, A},
    author={}
    SERIES = {},
    VOLUME = {},
 PUBLISHER = {},
      YEAR = {2024},
     PAGES = {},
      ISBN = {},
   MRCLASS = {},
  MRNUMBER = {},
MRREVIEWER = {},
       DOI = {},
       URL = {arXiv:2402.03552},
}

\bib{ABV}{article}
{
	title={The Langlands classification
and irreducible characters for real reductive groups}, 
	author={Adams, J},
    author={Barbasch, D},
    author={Vogan, D}
    SERIES = {Progress in Mathematics},
    VOLUME = {104},
 PUBLISHER = {Birkh\"auser Boston, Inc., Boston, MA},
      YEAR = {1992},
     PAGES = {xii+318},
      ISBN = {0-8176-3634-X},
   MRCLASS = {22-02 (22E47)},
  MRNUMBER = {1162533},
MRREVIEWER = {Brian\ E.\ Blank},
       DOI = {10.1007/978-1-4612-0383-4},
       URL = {https://doi.org/10.1007/978-1-4612-0383-4},
}






\bib{C}{article}
{
	title={A relative Langlands dual realization of $T^*(G/K)$ and derived Satake}, 
    Year={2026}
	author={Chen, T.-H},
    Note={Available at \url{https://arxiv.org/abs/2601.18022}}
}







\bib{CY}{article}
{
	title={Matsuki duality for loop groups}, 
	author={Chen, T.-H},
    author={Yi, L.},
    Year={2026}
    Note={Available at \url{https://arxiv.org/pdf/2604.15712}}
  
}

\bib{CY2}{article}
{
	title={Singularities of orbit closures in loop spaces of symmetric varieties II: twisted cases}, 
	author={Chen, T.-H},
    author={Yi, L.},
   Note={in preparation}
  
}


\bib{K}{article}
{
	title={$(B(G)$ for all local and global fields}, 
	author={Kottwitz, R},
    author={}
    YEAR={2014}
	Note={Available at \url{https://arXiv:1401.5728}},
}


\bib{V}{article}
{
	title={Branching to a maximal compact subgroups}, 
	author={Vogan, D},
    author={},
    author={}
    SERIES = {},
    VOLUME = {},
 PUBLISHER = {},
      YEAR = {2007},
     PAGES = {},
      ISBN = {},
   MRCLASS = {},
  MRNUMBER = {},
MRREVIEWER = {},
       DOI = {},
       URL = {},
}

\end{biblist}
\end{bibdiv}
\end{document}